\documentclass[]{interact}

\usepackage{epstopdf}
\usepackage{subfigure}

\usepackage{natbib}
\bibpunct[, ]{(}{)}{;}{a}{}{,}
\usepackage{bbm}
\usepackage[colorlinks=true, allcolors=blue]{hyperref} 
\usepackage{color}

\theoremstyle{plain}
\newtheorem{theorem}{Theorem}[section]
\newtheorem{lemma}[theorem]{Lemma}
\newtheorem{corollary}[theorem]{Corollary}
\newtheorem{proposition}[theorem]{Proposition}
\newtheorem{hyp}[theorem]{HYPOTHESIS}

\theoremstyle{definition}
\newtheorem{definition}[theorem]{Definition}
\newtheorem{example}[theorem]{Example}
\newtheorem{ex}[theorem]{Example}
\theoremstyle{remark}
\newtheorem{remark}{Remark}
\newtheorem{notation}{Notation}

 \DeclareMathAlphabet{\mathpzc}{OT1}{pzc}{m}{it}
 \DeclareMathAlphabet{\mathsfsl}{OT1}{cmss}{m}{sl}

  \newcommand{\V}{\mathcal{V}}
  
  \newcommand{\FH}{\mathfrak{H}}

\newcommand{\dif}{\mathrm{d}}

\newcommand{\R}{\mathbb{R}}

\newcommand{\abs}[1]{\left\vert#1\right\vert}
\newcommand{\set}[1]{\left\{#1\right\}}

 \newcommand{\innp}[1]{\langle {#1}\rangle}
 
 \newcommand{\1}{\mathbf{1}}
 \newcommand{\Rnum}{\mathbb{R}}

\newcommand{\E}{\mathbb{E}}

\begin{document}

\title{The properties of fractional Gaussian Process and their Applications.}

\author{
\name{Yong CHEN\textsuperscript{a},  and  Ying LI\textsuperscript{b},\thanks{CONTACT Ying LI, Email: liying@xtu.edu.cn.}}
\affil{ \textsuperscript{a}School of Mathematics and Statistics, Jiangxi Normal University, Nanchang, 330022, Jiangxi, China.  
\textsuperscript{b}School of Mathematics and Computional Science, Xiangtan University, Xiangtan, 411105, Hunan, China.}
}

\maketitle
\begin{abstract}
The process $(G_t)_{t\in[0,T]}$
  is referred to as fractional Gaussian process if the first-order partial derivative of the difference between its covariance function and that of the fractional Brownian motion $(B^H_t)_{t\in[0,T ]}$ is a normalized bounded variation function. We quantify the relation between the associated reproducing kernel Hilbert space of $(G)$ and that of  $(B^H)$. Seven types of Gaussian processes with non-stationary increments in the literature belong to it. In the context of applications, we demonstrate that the Gladyshev's theorem holds for this process, and we provide Berry-Ess\'{e}en upper bounds associated with the statistical estimations of the ergodic fractional Ornstein-Uhlenbeck process driven by it. The second application partially builds upon the idea introduced in \cite{BBES 23}, where they assume that $(G)$ has stationary increments. Additionally, we briefly discuss a variant of this process where the covariance structure is not entirely linked to that of the fractional Brownian motion.
\end{abstract}
\begin{amscode} 60G22; 60G15; 60F15; 62M09 \end{amscode}

\section{Introduction}\label{sec-int}
The motivation of this paper is to propose a precise definition of fractional Gaussian process and investigate its properties and applications. Although the terminology of fractional Gaussian process is used extensively in the literature such as \cite{Begyn 07, limeba, zhangxiao 17}, {a clear definition of fractional Gaussian process has not yet been established}. It is believed that commonly used Gaussian processes with nonstationary increments, such as sub-fractional Brownian motion, bi-fractional Brownian motion, generalized fractional Brownian motion, and generalized sub-fractional Brownian motion, fall under the category of fractional Gaussian process. This is because these Gaussian processes possess properties that are similar to those of fractional Brownian motion (fBm). 

{If we compare the covariance structures (covariance functions) of these Gaussian processes with that of fBm, we observe a common term:}
\begin{align*}
-\frac12\abs{t-s}^{2H}.
\end{align*}
Hence, it is natural to use the difference between {the covariance functions to assess how far is the fractional Gaussian process from the fBm.} This motivates us to define the fractional Gaussian process as follows:
\begin{definition}
The centred Gaussian process $(G_t)_{t\in[0,T ]}$ with $G(0)=0$  
is called a fractional Gaussian process if its covariance function $R(s, t)$ satisfies the following hypothesis: 
  \begin{hyp}\label{hypthe1}
The covariance function $R(s,t)=E (G_tG_s)$ satisfies that
\begin{enumerate}
  \item[($H_1$)] for any fixed $s \in[0,T]$, $R(s,t)$ is an absolutely continuous function on $t\in [0,T]$;
  \item[($H_2$)] for any fixed $t \in[0,T]$, the difference
\begin{align}
\frac{\partial R(s,t)}{\partial t} - \frac{\partial R^{B}(s,t)}{\partial t} \label{cha-key}
\end{align}
is a normalized bounded variation function on $s\in  [0,T]$, where $R^B(s,t)$ is the covariance function of fBm $(B^H_t)_{t\in[0,T ]}$ with Hurst parameter $H\in (0, \,1)$. 
\end{enumerate}
\end{hyp}
\end{definition}
A bounded variation function $F$ is said to be normalized if it is right continuous on $\R$ and its limits at $-\infty$ is zero, i.e., $F(-\infty)=0$, see \cite{Foll 99}. The assumption of the process $(G) $ being centred (zero mean) and $G(0)=0$ implies that for any $t \ge 0$, $R(0,t)=0$ and hence $$\frac {\partial R(0,t)}{\partial t} - \frac {\partial R^{B}(0,t)}{\partial t}=0.$$ 
Clearly, we can extend the value of the difference \eqref{cha-key} to be equal to zero in $(-\infty,\,0)$ and equal to the value at the endpoint $T$ in $(T,\infty)$. Thus, to check this bounded variation function is normalized on $s\in  [0,T]$, we only need to show that it is also right continuous on $s\in  [0,T]$.

This definition of fractional Gaussian process is an extension of that in \cite{Chenzhou2021, withdingzhen,  chengu2021}, where Hypothesis ($H_{2}$) is taken as the following special case:
\begin{enumerate}
  \item[($H_{2}'$)] For any fixed $t \in[0,T]$, the difference \eqref{cha-key}
  is an absolutely continuous function on $s\in [0,T]$.
 \end{enumerate}

The aforementioned Gaussian processes {satisfy ($H_{2}'$)}, please see Section~\ref{exmp section-1} for details. Here we would like to single out one example in which for any fixed $t \in[0,T]$, the difference \eqref{cha-key} 
 is a (pure discontinuous) step function on $s\in [0,T]$, as discussed in Section~\ref{sec2-1 new}.
\begin{example}\label{exmp lizi001}
Suppose that $H \in (0, \frac12)$. The Gaussian process $(G) $ has the covariance function 
 \begin{align} \label{cova func 01}
 R(s, t)=\frac12\big[ (\max(s,t))^{2H} - |t-s|^{2H}\big] .
 \end{align}
For further details, please refer to  Theorem 1.1 of \cite{Talarczyk2020}.
\end{example}

All kinds of assumptions regarding the second-order partial derivative  $\frac{\partial^2  R(s,t)}{\partial s\partial t}$ and their corresponding properties and applications can be traced back to a very long time ago, see \cite{Adler 81} for example. However,  Hypothesis~\ref{hypthe1} {is a novel contribution in our knowledge}.  In Section~\ref{mainresul}, we will present its properties and applications. We point out that the similar approach to bound the mixed second-order or four-order partial derivative of the difference $R(s,t)-R^B(s,t)$ also appeared in the literature \cite{Nor 11} and \cite{CGPP 2006} respectively.  Moreover, \cite{Viitasaari19} proposed the idea of bounding the mixed second-order partial derivative of the structure functions of $(G)$, denoted as 
$$\psi_G(s,t):=\E\big[(G_s-G_t)^2\big] .$$

The rest of the paper is organized as follows: In Section~\ref{mainresul} we present our main results: (i) We  quantify the relation between the associated reproducing kernel Hilbert space of $(G)$ and $(B^H)$; (ii) We discuss Gladyshev's theorem for the fractional Gaussian process $(G)$; (iii) We provide three Berry-Ess\'{e}en types upper bounds that are related to the statistical estimations of ergodic fractional Ornstein-Uhlenbeck (OU) process driven by the fractional Gaussian process $(G)$. Sections~\ref{sec-2} and~\ref{sec3 applications} are devoted to show the results (i) and (ii)-(iii) respectively. In Section~\ref{exmp section} we illustrate that the other six types Gaussian processes in the literature satisfy Hypothesis~\ref{hypthe1}. In Section~\ref{sec-dis} we study a type of Gaussian process whose covariance structure is not entirely linked to that of the fBm. In Section~\ref{appendx},  we present several auxiliary lemmas, including the integration by parts of signed Lebesgue-Stieltjes measure and the estimates of three cumulants of a Wiener chaos.

Except in Section~\ref{sec-dis},  we always assume that Hypothesis~\ref{hypthe1} holds throughout the paper.  In addition, the symbol $C$ stands for a generic constant, whose value can change from one line to another.

\section{Main results}\label{mainresul}	
Let $\FH$ and $\FH_1$ denote the associated reproducing kernel Hilbert spaces of the fractional Gaussian process $(G)$ and the fBm $(B^H)$ respectively, as discussed in Section~\ref{sec-2}. Furthermore, let $\mathcal{V}_{[0,T]}$ represent the set of bounded variation functions on the interval $[0,T]$.

\subsection{The relation between $\FH$ and $\FH_1$}\label{sec2-1 new}

The first result of this paper is to quantify the relationship between $\FH$ and $\FH_1$.
 \begin{theorem}\label{prop 2-1}
Let  $f,\, g\in \mathcal{V}_{[0,T]}$. If $g(\cdot)$ has no common discontinuous points with the difference function $\frac {\partial R(\cdot,t)}{\partial t} - \frac {\partial R^{B}(\cdot,t)}{\partial t} $ on $[0,T]$ for each fixed $t\in [0,T]$, then
  \begin{align} \label{diyierdengshi-yiban}
 \langle f,\,g \rangle_{\FH}-  \langle f,\,g \rangle_{\FH_1}&=\int_{[0,T]^2}  {f}(t)g(s)  \dif t \,\mu(t,\dif s),
 \end{align} where $\mu(t,\cdot)$ is the Lebesgue-Stieltjes measure of the normalized bounded variation function $\frac {\partial R(\cdot,t)}{\partial t} - \frac {\partial R^{B}(\cdot,t)}{\partial t} $ on  $[0,T]$ for each fixed $t\in [0,T]$. In particular, if $g$ is continuous, then the identity \eqref{diyierdengshi-yiban} holds.
 \end{theorem}
Proof of Theorem~\ref{prop 2-1} will be given in Section~\ref{sec-2}. It is evident that if Hypothesis~\ref{hypthe1} ($H_2$) is substituted by ($H_2'$), the condition for $g$ in Theorem~\ref{prop 2-1} is always satisfied, see \cite{Foll 99, tao2011}.  We restate it as a corollary.

 \begin{corollary} \label{cor2-2}
If the covariance function $R(t,s)$ also satisfies Hypothesis~\ref{hypthe1} ($H_2'$),  then  $\forall f,\, g\in \mathcal{V}_{[0,T]}$,
  \begin{align} \label{diyierdengshi}
 \langle f,\,g \rangle_{\FH}-  \langle f,\,g \rangle_{\FH_1}&=\int_{[0,T]^2}  {f}(t)g(s)\frac{\partial}{\partial s}\Big(\frac {\partial R(s,t)}{\partial t} - \frac {\partial R^{B}(s,t)}{\partial t}\Big)  \dif t \dif s.
 \end{align} 
In addition, if the following Hypothesis~($H_3$) is also satisfies:
\begin{enumerate}  \item[($H_3$)] there exists a positive constant $C$ independent {of} $T$ such that  \begin{align}\label{bijiao 00} \Bigg| \frac{\partial}{\partial s}\Big(\frac {\partial R(s,t)}{\partial t} - \frac {\partial R^{B}(s,t)}{\partial t}\Big)\Bigg|  & \le  C   (ts )^{H-1},\quad a.s. 
\end{align} 
holds.
\end{enumerate}
Then we have
\begin{equation}\label{inequality 29}
\abs{\langle f,\,g \rangle_{\mathfrak{H}} - \langle f,\,g \rangle_{\mathfrak{H}_1}}\leq  C \nu(\abs{f})\nu(\abs{g}),
\end{equation} 
where 
$\nu(f)=\int_0^T f\dif\nu$ and $\nu(\dif x)=x^{H-1}\dif x$ and $C$ is a positive constant independent {of} $T$. 
\end{corollary} 
The inequality \eqref{inequality 29} follows from the modulus inequality of integration.

In the rest of this subsection, we will use Example~\ref{exmp lizi001}  to illustrate how to deal with the case when the difference \eqref{cha-key}
 is not absolutely continuous. 
It is clear that for the process $(G)$ of Example~\ref{exmp lizi001} and for any fixed $t\in [0,T]$, the difference
\begin{align}\label{cha-hanshu-juti}
  \frac{\partial R(s,t)}{\partial t  }- \frac{\partial  R^B(s,t)}{\partial t  }=\left\{
      \begin{array}{ll}
0, & \quad 0<s\le t,\\
-H t^{2H-1}, & \quad t<s\le T,
 \end{array}
\right. 
\end{align} 
is a step function of $s\in [0,T]$. Hence Hypothesis ($H_2'$) fails. However, this step function is a normalized bounded variation function. Hence, it follows from Theorem~\ref{prop 2-1} that the following quantified relation between $\FH$ and $\FH_1$ holds: if $f,\, g\in \mathcal{V}_{[0,T]}$ and $g(\cdot)$ has no common discontinuous points with the difference function \eqref{cha-hanshu-juti} then the following equation holds: 
 \begin{align}\label{guanjiandengshi 000-new1}
 \langle f,\,g \rangle_{\FH}-  \langle f,\,g \rangle_{\FH_1}&=-H \int_0^T {f}(t) g(t)t^{2H-1}\dif t . \end{align}
 A direct corollary of the identity \eqref{guanjiandengshi 000-new1} is the following interesting fact: if the intersection of the supports of $f$ and $g$ has zero Lebesgue measure, then 
 \begin{align} \label{innp fg3-zhicheng0-00}
\langle f,\,g \rangle_{\FH}=\langle f,\,g \rangle_{\FH_1}=\int_{[0,T]^2}  f(t)g(s) \frac{\partial^2 R^{B}(t,s)}{\partial t\partial s} \dif t   \dif s. 
\end{align} 
In fact, this identity can also be derived from the following basic relationship: for $0 \leq a < b \leq c < d \leq T$, we have 
  \begin{align} \label{innp fg3-zhicheng0 tebie 0000}
  \mathbb{E}[(G_b-G_a)(G_d-G_c)] = \mathbb{E}[(B^H_b-B^H_a)(B^H_d-B^H_c)] .
\end{align}
Clearly, the above two identities do not hold any more if the intersection of the supports of $f$ and $g$ has nonzero Lebesgue measure.



\subsection{Application to Gladyshev's theorem}
In this and the following subsection, we will present two applications of Theorem~\ref{prop 2-1} or Corollary~\ref{cor2-2}. Unless a special announcement, we always suppose that either the fractional Gaussian process $(G)$ satisfies Hypothesis~ ($H_2'$)-($H_3$) or is given as in Example~\ref{exmp lizi001} .
\begin{theorem}\label{Gladyshev theorem}
The Gladyshev's theorem holds, i.e., the relation
\begin{align}\label{gladshev thm00}
\lim_{n\to\infty}\left(\frac{2^n}{T}\right)^{2H-1}\sum_{k=1}^{2^n} \left[G\left(\frac{kT}{2^n}\right)-G\left(\frac{(k-1)T}{2^n}\right)\right]^2= T
\end{align}
holds with probability 1.
\end{theorem}
Theorem~\ref{Gladyshev theorem} is new in the following sense: the continuity of $\frac{\partial^2 R(s,t)}{\partial s\partial t}$ on $U:=\set{(s,t)\in(0,T]\times (0,T];\,s\neq t}$ is not assumed when $(G)$ satisfies Hypothesis ($H_2'$)-($H_3$), although Example~\ref{exmp lizi001} satisfies all the conditions of Gladyshev's Theorem, \cite{Nor 11}. The reason for limiting the generality is that when dealing with a fractional Gaussian processes that ($H_2'$) fails but ($H_2$) is valid, it is preferable to find a suitable upper bound for the total variation function of the difference \eqref{cha-key} in each specific case. Although we do not find other unartifical examples {in the literature where} ($H_2'$) fails but ($H_2$) is valid, we believe that the same idea can be applied to other fractional Gaussian processes in the future.

\subsection{Application to the statistical estimation of ergodic fractional Ornstein-Uhlenbeck process}
The second application is the statistical estimation of the ergodic OU process $(\zeta_t )_{t \in [0,T]}$ which is defined by the stochastic differential equation
\begin{align}
   \dif \zeta_t &=-\theta \zeta_t \dif t +\dif  G_t, \quad \zeta_0 = 0. \label{OU dingyi} 
\end{align} 
Suppose that the whole trajectory of  $(\zeta )$ is observed continuously on $[0, T]$. When $\theta>0$,  i.e., in the ergodic case, the informal least squares estimator (LSE) and the moment estimator of the parameter $\theta$ are respectively constructed from the continuous observations as follows:
 \begin{align}
\hat{\theta}_T&=-\frac{``\int_0^T  \zeta_t\mathrm{d} \zeta_t"}{\int_0^T  \zeta_t^2\mathrm{d} t}=\theta-\frac{\int_0^T  \zeta_t\mathrm{d}G_t}{\int_0^T  \zeta_t^2\mathrm{d} t}, \label{hattheta}\\
\tilde{\theta}_{T}&=\Big( \frac{1}{ {H} \Gamma(2H ) T} \int_0^T  \zeta_t^2\mathrm{d} t \Big)^{-\frac{1}{2H}},\label{theta tilde formula}
\end{align}
where the integral with respect to $G_t$ is interpreted in the Skorohod sense, also known as a divergence-type integral, see \cite{HN 10, HNZ 19}. 

The aim of this subsection is to derive the Berry-Ess\'{e}en type upper bound for the two estimators (In fact, the LSE is informal and not a real statistical estimator). We utilize the methods presented in  \cite{BBES 23} and \cite{Chenzhou2021} respectively. {The main idea is to establish a connection between the Berry-Ess\'{e}en type upper bound of the two estimators and that of the following} chaos process $W_T$:
\begin{notation}
Let $Z\sim N(0,\sigma^2_B)$ and consider a chaos process:  
\begin{align} 
W_T&=\frac{1}{\sqrt {T}}\int_{0}^T \big( \zeta_t^2 -\E[  \zeta_t^2] \big)\dif t,\label{wt dyi} \end{align} 
which belongs to the second Wiener chaos (i.e., Wiener-It\^o multiple integrals) with respect to $(G)$, where 
\begin{align}
\sigma^2_B&={\theta^{-1-4H}} {\big(H\Gamma(2H)\big)^2}  \sigma^2_H, \label{sb2 dyi}\\ 
\sigma^2_H&= (4H-1) +   \frac{2 \Gamma(2-4H)\Gamma(4H)}{\Gamma(2H)\Gamma(1-2H)}.
\end{align}
\end{notation}

\begin{theorem}\label{coro key guanj}
When $H\in (0,\frac12)$, the total variation distance between the random variables $W_T$ and $Z$ satisfies
\begin{align}\label{qbcha juli}
d_{TV}(W_T,\,Z)\le C\times \frac{1}{\sqrt{T}}.\end{align}
\end{theorem}
\begin{remark}
The case of $H\in (\frac12, \frac34]$ has been treated in \cite{Chenzhou2021}.
\end{remark}


\begin{theorem}\label{B-E bound thm} 
Suppose the fractional Gaussian process $(G)$ satisfies Hypothesis ($H_2'$)-($H_3$).
Let $\Phi(z)$ denote the standard normal distribution function. 
If $H\in (0,\,\frac12)$ then there exists a positive constant $C_{\theta, H}$ such that when $T$ is sufficiently large,
\begin{equation}\label{b-e bound 44}
\sup_{z\in \Rnum}\abs{P(\sqrt{\frac{4H^2 T}{\theta \sigma^2_{H}}} (\tilde{\theta}_T-\theta )\le z)-\Phi( z)}\le\frac{ C_{\theta, H}}{\sqrt{T  }},
\end{equation}and 
\begin{equation}\label{b-e bound 34}
\sup_{z\in \Rnum}\abs{P(\sqrt{\frac{T}{\theta \sigma^2_{H}}} (\hat{\theta}_T-\theta )\le z)-\Phi( z)}\le\frac{ C_{\theta, H}}{\sqrt{T} }.
\end{equation}
\end{theorem}
\begin{remark}
\begin{enumerate}
  \item[(1)] When the Gaussian process $(G)$ in equation \eqref{OU dingyi} degenerates to the fBm, the same problem has been studied in the literature \cite{chenkl2020, chenl2021},  and the improving upper bound has recently been obtained in \cite{BBES 23, chengu2023}. 
    \item[(2)] Theorem~\ref{B-E bound thm} improves the results of \cite{chengu2021} where they show that when $H\in (0,\frac38)$, the upper bound is $\frac{C}{T^{\frac12 \wedge (\frac32-4H)}}$. In fact, this is one of the motivations of the present paper.
    \item[(3)]  The case of $H\in (\frac12, \frac34]$ has been treated in \cite{Chenzhou2021}. The methods used are based on \cite{Kim17} when $H\in (\frac12, \frac34]$ and on \cite{BBES 23} when $H\in (0,\frac12)$, respectively. The two methods are different significantly.
    \end{enumerate}
\end{remark}

When the Gaussian process $(G)$ is given in Example~\ref{exmp lizi001}, we cannot obtain a good Berry-Ess\'{e}en upper bound for the moment estimate $\tilde{\theta}_{T}$ due to  the bound of 
\begin{align*}
 \abs{\frac{1}{ \sqrt {T}}\int_{0}^T  \E[  \zeta_t^2-\eta^2_t] \dif t  }\le CT^{2H-\frac12}
\end{align*}
not being good enough, where $\eta_t$ is given in \eqref{OU dingyi duibi}. However, for the  informal LSE  $\hat{\theta}_{T}$, we can give an upper bound of $$ \frac{C}{ {T}^{ (1-2H)\wedge\frac12} },$$ as shown in Proposition~\ref{sijiejun juli qidian} and the proof of Theorem~\ref{B-E bound thm}. These facts illustrate the essential difference between the fractional Gaussian process $(G)$ given in Example~\ref{exmp lizi001} and the standard fBm, despite their similar covariance functions. We will investigate this topic further in future works.
\section{The quantified relation between the associated reproducing kernel Hilbert spaces of  $G$ and $B^H$}\label{sec-2}
Let $G = \set{G_t, t\in [0,T ]}$ be a continuous centred Gaussian process with covariance function $R(s,t)$
defined on a complete probability space $(\Omega, \mathcal{F}, P)$. The filtration $\mathcal{F}$ is generated by the Gaussian family $G$. Suppose in addition that the covariance function $R(s,t)$ is continuous.
Let $\mathcal{E}$ denote the space of all real-valued step functions on $[0,T]$. The associated reproducing kernel  Hilbert space $\mathfrak{H}$ is defined
as the closure of $\mathcal{E}$ endowed with the inner product
\begin{align*}
\innp{\mathbbm{1}_{[a,b]},\,\mathbbm{1}_{[c,d]}}_{\FH}=\E\big(( G_b-G_a) ( G_d-G_c) \big).
\end{align*}
By abuse of notation, we use the same symbol 
\begin{align}
G=\left\{G(h)=\int_{[0,T]}h(t)\dif G_t, \quad h \in \mathfrak{H}\right\}\end{align} to denote the isonormal Gaussian process on the probability space $(\Omega, \mathcal{F}, P)$, indexed by the elements in the Hilbert space $\mathfrak{H}$. In other words, $G$ is a Gaussian family of random variables such that 
\begin{align}\label{G extension defn}
 \mathbb{E}(G(h)) = 0, \qquad \mathbb{E}(G(g)G(h)) = \langle g, h \rangle_{\mathfrak{H}}, \quad
\forall g, h \in \mathfrak{H}.   
\end{align}Here the second identity is called It\^{o}'s isometry.
In particular, when $G$ is exact the fBm $B^H$, we denote by $\mathfrak{H}_1$ the associated Hilbert space with the inner product $\innp{\cdot,\,\cdot}_{\FH_1}$. 

For each $g\in \mathcal{V}_{[0,T]}$, {we define $\nu_{g}$ as the restriction of the Lebesgue-Stieljes signed measure associated with $g^0$ to $([0,T ], \mathcal{B}([0,T ]))$. Here, $g^0$ is defined as}:
\begin{equation*}
g^0(x)=\left\{
      \begin{array}{ll}
 g(x), & \quad \text{if } x\in [0,T],\\
0, &\quad \text{otherwise}.     
 \end{array}
\right.
\end{equation*} 
{The idea of using the measure $\nu_{g}$ is derived from \cite{Jolis 2007}. In this paper, we begin with a measure decomposition of $\nu_{g}$.} Firstly, we can decompose $g^0(x)$ as the difference of two bounded variation functions:
\begin{align}\label{basic decomp}
g^0(x)=g_1^0(x)-g_2^0(x),
\end{align}
where 
\begin{equation*}
g_1^0(x)=\left\{
      \begin{array}{ll}
      0,&\quad \text{if } x<0,\\
 g(x), & \quad \text{if } x\in [0,T],\\
g(T), &\quad \text{if } x>T,    
 \end{array}
\right.
\end{equation*}
and 
\begin{equation*}
g_2^0(x)=\left\{
      \begin{array}{ll}
      0,&\quad \text{if } x<T,\\
g(T), &\quad \text{if } x>T.    
 \end{array}
\right.
\end{equation*}
By the near uniqueness of the Lebesgue-Stieltjes measures of bounded variation functions \cite[p.159]{tao2011}, we can assume  that $g^0_1$ and $g_2^0$ are normalized. Therefore, we have the following equation:
\begin{align}\label{ceducha}
\mu_{g^0} =\mu_{g^0_1}-\mu_{g^0_2}.
\end{align} 
It is evident  that 
\begin{equation}
\mu_{g^0_2}=g(T) \delta_T(\cdot),\label{g02 cedu}
\end{equation} 
where $\delta_T(\cdot)$ represents the Dirac measure at $T$. 


The measure decomposition \eqref{ceducha} allows us to derive an integration by parts formula (see Lemma~\ref{partial integral}). This, combined with a result from \cite{Jolis 2007}, leads us to the following proposition, which serves as another starting point for this paper.
\begin{proposition}\label{prop 2-1010}
Let  $f,\, g\in \mathcal{V}_{[0,T]}$. If the covariance function $R(t,s)$ satisfies Hypothesis~\ref{hypthe1} ($H_1$), then 
  \begin{align} \label{innp fg3}
\langle f,\,g \rangle_{\FH}=-\int_{[0,T]^2}  f(t) \frac{\partial R(t,s)}{\partial t} \dif t  \nu_{g}(\dif s).
\end{align} 
\end{proposition}
\begin{remark}
For example, we consider functions of the form $g=h\cdot \1_{[a,b]}$, where $0\le a<b\le T$ and $h$ is a differentiable function. In this case, the Lebesgue-Stieljes signed measure $\nu_{g}$ on  $([0,T ], \mathcal{B}([0,T ]))$ can be expressed as:
\begin{equation}\label{jieshi01}
  \nu_g(\dif x)=  h'(x)\cdot \1_{[a,b]}(x)\dif x+  h(x)\cdot \big(\delta_a(x)-\delta_b(x)\big)\dif x,
\end{equation}
where $\delta_a(\cdot)$ represents the Dirac delta function centered at a point $a$.  
\end{remark}
\begin{proof}
It follows from \cite{Jolis 2007} that the following inner product formula of $\FH$ holds: 
 \begin{align}\label{jolis 00}
 \langle f,\,g \rangle_{\FH}=\int_{[0,T]^2}  R(s,t) \nu_{f}( \dif t)  \nu_{g}(\dif s),\qquad \forall f,\, g\in \mathcal{V}_{[0,T]}.
 \end{align}
Under Hypothesis~\ref{hypthe1} ($H_1$), we apply Lemma~\ref{partial integral}, the integration by parts formula, to the absolutely continuous function $R(s,\,\cdot)$ on $t\in [0,T]$ and $f(\cdot)\in \V_{[0,T]}$ to obtain:
 \begin{align}
  \langle f,\,g \rangle_{\FH}&=\int_{[0,T]}  \nu_{g}(\dif s)\int_{[0,T]} R(s,t) \nu_{f}( \dif t)  =-\int_{[0,T]}  \nu_{g}(\dif s)\int_{[0,T]} \frac{\partial R(s,t)}{\partial t} {f}(t) \dif t.\label{fengkong 000}
 \end{align}
\end{proof}

 {\bf Proof of Theorem~\ref{prop 2-1}}:  
First, since $R^B(s,t)$, the covariance function of fBm $(B^H)$ with $H\in (0,1)$, also satisfies Hypothesis~\ref{hypthe1}  ($H_1$), it follows from Proposition~\ref{prop 2-1010} that
\begin{align*}
  \langle f,\,g \rangle_{\FH_1}&=-\int_{[0,T]}  \nu_{g}(\dif s)\int_{[0,T]} \frac{\partial R^B(s,t)}{\partial t} {f}(t) \dif t.
 \end{align*} 
Subtracting the above equation displayed from the identity \eqref{fengkong 000}, and using Fubini's theorem, we imply that 
 \begin{align*}
 \langle f,\,g \rangle_{\FH}-  \langle f,\,g \rangle_{\FH_1}&=\int_{[0,T]}  \nu_{g}(\dif s)\int_{[0,T]} \Big(\frac{\partial R^B(s,t)}{\partial t} -  \frac{\partial R(s,t)}{\partial t}\Big){f}(t) \dif t\notag\\
 &=\int_{[0,T]}  {f}(t) \dif t\int_{[0,T]} \Big(\frac{\partial R^B(s,t)}{\partial t} - \frac{\partial R(s,t)}{\partial t}\Big) \nu_{g}(\dif s).
 \end{align*}
 Under Hypothesis~\ref{hypthe1} ($H_2$), we apply Lemma~\ref{partial integral} to the normalized bounded variation function $\frac {\partial R(\cdot,t)}{\partial t} - \frac {\partial R^{B}(\cdot,t)}{\partial t}$ on $s\in  [0,T]$ and the function $g(\cdot)\in \V_{[0,T]}$ to derive the desired equation \eqref{diyierdengshi-yiban}:
 \begin{align}\label{guanjiandengshi 000}
 \langle f,\,g \rangle_{\FH}-  \langle f,\,g \rangle_{\FH_1}&=\int_{[0,T]}  {f}(t) \dif t\int_{[0,T]} g(s) \mu(t,\dif s).
 \end{align}  
 {\hfill\large{$\Box$}} 

\begin{proposition}\label{prop 2-101}
Let  $f,\, g\in \FH$. 
If the covariance function $R(t,s)$ also satisfies Hypothesis ($H_2'$) and if the intersection of the supports of $f,\,g$ is of Lebesgue measure zero, then 
 \begin{align} \label{innp fg3-zhicheng0}
\langle f,\,g \rangle_{\FH}=\int_{[0,T]^2}  f(t)g(s) \frac{\partial^2 R(t,s)}{\partial t\partial s} \dif t   \dif s. 
\end{align} 
In particular, when $0 \leq a < b \leq c < d \leq T$, we have
  \begin{align} \label{innp fg3-zhicheng0 tebie}
  \mathbb{E}[(G_b^H-G_a^H)(G_d^H-G_c^H)] = \int^b_a\dif t\int^d_c \frac{\partial^2 R(t,s)}{\partial t\partial s} \dif s.
\end{align} 
\end{proposition}
\begin{proof}
The claim \eqref{innp fg3-zhicheng0} is a direct corollary of the identity \eqref{diyierdengshi} and the two well-known facts that  $\mathcal{V}_{[0,T]}$ is dense in both $\FH$ and $\FH_1$ (see \cite{Jolis 2007}) and that when the intersection of the supports of $f,\,g$ is of Lebesgue measure zero, we have 
\begin{align}
\langle f,\,g \rangle_{\FH_1}&= H(2H-1)\int_{[0,T]^2}\,f(t)g(s)|t-s|^{2H-2}\dif t \dif s,\qquad \forall f,g\in \FH_1.
\end{align}The claim \eqref{innp fg3-zhicheng0 tebie} is trivial.
\end{proof}

\section{Proofs of Theorems~\ref{Gladyshev theorem}, \ref{coro key guanj}, \ref{B-E bound thm}}\label{sec3 applications}
Recall that for these three theorems, we always assume that either the fractional Gaussian process $(G)$ satisfies Hypothesis~ ($H_2'$)-($H_3$) or is given as in Example~\ref{exmp lizi001}. 
\subsection{Gladyshev theorem}\label{Gladyshev th}
The following proposition is used to verify the assumptions (c) of Gladyshev's Theorem, see \cite{Nor 11}. 
\begin{proposition}\label{coro 2-2}
There exists a positive constant $C$ independent {of} $T$ such that $d_G$, the canonical metric for $(G)$, satisfies  
\begin{align}\label{gima2fang jie}
d_G(s,t):=\left\{\E\left[(G_s-G_t)^2\right]\right\}^{\frac12}\le C\abs{s-t}^{H}.
\end{align} 
Denote $$b(s,t):=\E\left[(G_s-G_t)^2\right]- \E\left[(B^H_s-B^H_t)^2\right].$$ Then for all $\epsilon>0$,
\begin{align}\label{h jixian}
\lim_{h\downarrow 0}\frac{\sup\set{\abs{b(t-h, t)}:\, t\in (\epsilon, T]}}{h^{2H}}=0.
\end{align}
\end{proposition}
\begin{proof}
Case (i): When $(G)$ satisfies Hypothesis~ ($H_2'$)-($H_3$). 
Suppose that $s<t$ and set $f(\cdot)=\mathbbm{1}_{[s,t]}(\cdot)$,
 then it follows from It\^o's isometry and the inequality \eqref{inequality 29} that
 \begin{align}
\abs{\E\left[(G_s-G_t)^2\right] -\E\left[(B^H_s-B^H_t)^2\right]}&=\abs{\innp{f,\,f}_{\FH}- \innp{f,\,f}_{\FH_1}}\notag\\
&\le C \left(\int_{[s,t]}  u^{H-1}\dif u  \right)^2 =C t^{2H}\left(1- \left(\frac{s}{t}\right)^H  \right)^2 \label{zhongjianbuds}\\
&\le C\abs{s-t}^{2H},\notag
 \end{align}where in the last inequality we have used the fact $1-x^H \leq (1-x)^H$ for any $x \in [0, 1]$ and $H\in (0,1)$.  
 It is clear that for the fBm $(B^H)$, we have 
 \begin{align*}
 \E\left[(B^H_s-B^H_t)^2\right]=\abs{s-t}^{2H}.
 \end{align*}
 Then the claim \eqref{gima2fang jie} follows from the triangle inequality for absolute values.
 
Next, it follows from the inequality \eqref{zhongjianbuds} that there exists a positive constant $C$ independent of $T$ such that  for all $h,\,\epsilon>0$,
 \begin{align*}
 \sup\set{\abs{b(t-h, t)}:\, t\in (\epsilon, T]}&\le C \times\left(\sup\set{t^H- (t-h)^H:\, t\in (\epsilon, T]}\right)^2\\
 &=C \times \left[{\epsilon}^H- (\epsilon-h)^H\right]^2,
 \end{align*}
 where in the last line we use the monotonicity of the function $\phi(t)=t^H- (t-h)^H$ on $t\in[\epsilon, T]$. By L'H\^opital's rule, we have 
\begin{align*}
 0\le \lim_{h\downarrow 0}\frac{\sup\set{\abs{b(t-h, t)}:\, t\in (\epsilon, T]}}{h^{2H}}\le C \times \left(\lim_{h\downarrow 0} \frac{{\epsilon}^H- (\epsilon-h)^H}{h^H}\right)^2=0
\end{align*}
and the claim follows.
 
Case (ii): When $(G)$ is given as in Example~\ref{exmp lizi001}. By the identity \eqref{guanjiandengshi 000-new1} or the covariance function \eqref{cova func 01}, we have that \begin{align}
\abs{\E\left[(G_s-G_t)^2\right] -\E\left[(B^H_s-B^H_t)^2\right]}&=H \int_{[s,t]}  u^{2H-1}\dif u   =\frac12 t^{2H}\left(1- \left(\frac{s}{t}\right)^{2H}\right)  \label{zhongjianbuds-new}\\
&\le \frac12\abs{s-t}^{2H},\notag
 \end{align} since $H\in (0,\frac12)$. Hence, the inequality \eqref{gima2fang jie}  holds. By the identity \eqref{zhongjianbuds-new}, we can show the limit \eqref{h jixian} in the same vein as Case (i).
 \end{proof}
 
{\bf Proof of Theorem~\ref{Gladyshev theorem}:} Case (i): When $(G)$ satisfies Hypothesis~ ($H_2'$)-($H_3$). 
The idea is to verify all the assumptions (a)-(c) of Gladyshev's Theorem, \cite{Nor 11}. Form its proof \cite[p.30]{Nor 11}, we find out firstly that the assumption of $R(s,t)$ having a continuous mixed second-order
partial derivative on $U:=\set{(s,t)\in(0,T]\times (0,T];\,s\neq t}$ can be relaxed to the inner product representation formula \eqref{innp fg3-zhicheng0 tebie}, and that the upper bound of  $\Bigg| \frac{\partial^2 R(s,t)}{\partial s  \partial t}  \Bigg| $ can be weaken to be valid a.e., see \eqref{piandaoshu shangjie} below for details.
  
Next, the assumption (a) of Gladyshev's Theorem, \cite{Nor 11}, is valid since $(G)$ has mean zero. Third, the covariance function $R(s,t)$ is continuous on $[0,T]\times [0,T]$. In fact, the identity \eqref{diyierdengshi} has a special case:
\begin{align} 
R(s, t)-R^B(s,t)&=\int_0^t \dif u \int_0^s\frac{\partial}{\partial v}\Big(\frac {\partial R(v,u)}{\partial u} - \frac {\partial R^{B}(v,u)}{\partial u}\Big)  \dif v,  \end{align} 
That is to say,
 \begin{align*}
 R(s,t)=R^B(s,t)+\int_0^t \dif u \int_0^s\frac{\partial}{\partial v}\left(\frac {\partial R(v,u)}{\partial u} - \frac {\partial R^{B}(v,u)}{\partial u}\right)  \dif v,
 \end{align*}
 which implies that $R(s,t)$ is continuous on $[0,T]\times [0,T]$ since $R^B(s,t)$ is so.
The inequality \eqref{bijiao 00}, together with the triangle inequality for absolute values,  implies that the inequality 
\begin{align}\label{piandaoshu shangjie}
\left| \frac{\partial^2 R(s,t)}{\partial s  \partial t}  \right|  &\le \left| \frac{\partial^2 R^B(s,t)}{\partial s  \partial t}  \right| +   C  \times (ts )^{H-1}
=H\abs{2H-1}\abs{s-t}^{2H-2}+  C   (ts )^{H-1}\quad a.e.
\end{align}
holds on $U:=\set{(s,t)\in(0,T]\times (0,T];\,s\neq t}$. 
Moreover, denote \begin{align*} f_1(s,t)={H } {\abs{s-t}^{2H-2} } ,\qquad f_2(s,t)={C } {(  st)^{H-1}}.
\end{align*}Then both the functions $f_i(s,t),\,i=1,2$, satisfy the following inequality (see the proof of Corollary 1 of \cite{Nor 11}):
\begin{align*}
\sum_{k=2}^{2^n-2}\sum_{j=k+2}^{2^n} {(\sup\set{f_i(s,t):\, (s,t)\in ((j-1)T,jT] \times( (k-1)T,kT] })^2}
\le CT^{4(H-1)}2^{n(4H-2)}.
\end{align*}
Hence, all the assumptions (b) of Gladyshev's Theorem \cite{Nor 11} are verified.  Finally, it follows from Proposition~\ref{coro 2-2} that the assumptions (c) of Gladyshev's Theorem \cite{Nor 11} are satisfied. Hence the claim follows.  
 
Case (ii): When $(G)$ is given as in Example~\ref{exmp lizi001}. It is evident that the mixed second-order partial derivative of $(G)$ on $U:=\set{(s,t)\in (0,T]\times(0,T],\,s\neq t}$ is equal to that of the fBm $(B^H)$. Therefore, the assumptions (b) of Theorem 1 of \cite{Nor 11} are satisfied. Proposition~\ref{coro 2-2} implies that the assumptions (c) of Theorem 1 of \cite{Nor 11} are also met. Consequently, Gladyshev's Theorem holds for  $(G)$ in Example~\ref{exmp lizi001}.
 
{\hfill\large{$\Box$}}

\subsection{Berry-Ess\'{e}en bound of the informal LSE and moment estimator for the OU process driven by fractional Gaussian processes}

Denote 
\begin{align}
    \zeta_t = \int_0^t e^{-\theta(t-u)}\dif G_u,\quad    \eta_t = \int_0^t e^{-\theta(t-u)}\dif B^H_u,
 \label{zeta T jifen}
\end{align}
where $\theta>0$ is a constant.  Then $(\zeta_t)_{t\in [0,T]},\,(\eta_t)_{t\in [0,T]}$ are respectively the ergodic OU process defined by the stochastic differential equations \eqref{OU dingyi} and 
\begin{align}
   \dif \eta_t &=-\theta \eta_t \dif t +\dif  B^H_t, \quad \eta_0 = 0.\label{OU dingyi duibi}
\end{align}  

First, from the quantified inequality relation \eqref{inequality 29}, we can compare the covariance function of $(\zeta)$ with that of $(\eta)$. 
\begin{proposition}\label{key prop bijiao}
Denote \begin{align}\label{rho12}
\rho_1(t,s)=\E[ \zeta_t \zeta_s],\quad \rho_2(t,s)=\E[\eta_t\eta_s].\end{align}
 Then there exists a positive constant $C$ independent of $T$ such that $\forall 0\le s<t\le T$,
\begin{align}
\abs{\rho_2(t,s)}&\le C\big(1\wedge(t-s)^{2(H-1)}\big); \label{rho2 guji}\end{align}
and when $(G)$ satisfies Hypothesis~ ($H_2'$)-($H_3$), \begin{align}\abs{\rho_1(t,s)-\rho_2(t,s)}&\le C \big(1\wedge s^{2(H-1)} \wedge (t-s)^{H-1}\big); \label{rho21ts cha bds}
\end{align}and when $(G)$ is given as in Example~\ref{exmp lizi001},
\begin{align}
0\le \rho_2(t,s)-\rho_1(t,s) \le C e^{(s-t)\theta} \big(1\wedge s^{2H-1} \big) \label{rho21ts cha bds-new case}.
\end{align} 
\end{proposition} 
\begin{proof}
First,  for any $0\le s<t\le T$,
we rewrite the OU process \eqref{OU dingyi} as follows:
\begin{align*}
 \zeta_t&= \int_0^t e^{-\theta(t-u)}\dif G_u =e^{\theta(s-t)}\zeta_s+ \int_s^t e^{-\theta(t-u)}\dif G_u.
 \end{align*} 
 Then we have that  for any $0\le s<t\le T$,
 \begin{align*} \zeta_t- \zeta_s &=(e^{\theta(s-t)}-1) \zeta_s+ G(e^{-\theta(t-\cdot)} \mathbf{1}_{[s,t]} (\cdot)).
\end{align*}
It follows from It\^{o}'s isometry and the identity \eqref{innp fg3-zhicheng0} that  for any $0\le s<t\le T$,
\begin{align}
\rho_1(t,s)&=\E[ \zeta_t \zeta_s]=\E[ (\zeta_t- \zeta_s)\zeta_s]+\E[ \zeta_s^2] \notag \\
&=\big[(e^{(s-t)\theta}-1)+1\big]\E[ \zeta_s^2]+\E\big[G(e^{-\theta(t-\cdot)} \mathbf{1}_{[s,t]} (\cdot))G(e^{-\theta(s-\cdot)} \mathbf{1}_{[0,s]} (\cdot))  \big]\notag  \\
&=e^{(s-t)\theta} \E[ \zeta_s^2]+\int_0^s e^{ ({u}-s)\theta}  \dif  {u} \int_s^t e^{ ({v}-t)\theta}\, \frac{\partial ^2 R(u,v)}{\partial u\partial v}\dif  {v}.\label{rho1ts bds}
\end{align} 
Denote $\alpha_H =H(2H-1)$.  A special case of \eqref{rho1ts bds} gives that for any $0\le s<t\le T$,
\begin{align} 
\rho_2(t,s)&=\E[ \eta_t \eta_s]=e^{(s-t)\theta}\E[ \eta_s^2]+\alpha_H \int_0^s e^{ ({u}-s)\theta }  \dif  {u} \int_s^t e^{ ({v}-t)\theta}  \abs{u-v}^{2H-2}\dif  {v}. \label{rho2ts bds} 
\end{align}

Next, Theorem 3.1 of \cite{ES 23}, or see \cite{BBES 23}, implies that \begin{align*} \sup_{t\ge 0} \E[\eta_t^2]<\infty,\end{align*}
which, together with the identity \eqref{rho2ts bds}, implies that for any $0\le s<t\le T$,
\begin{align}\label{qidan buds00}
\abs{\rho_2(s,t)}&\le C \Big[e^{(s-t)\theta}+\int_0^s e^{( {u}-s)\theta}  \dif  {u} \int_s^t e^{( {v}-t)\theta}(  {v}- {u})^{2H-2}\dif  {v}\Big].\end{align} It is clear that for any $0\le s<t$,
\begin{align*}e^{(s-t)\theta}<C \big(1\wedge(t-s)^{2(H-1)}\big)  .
\end{align*}
By change of variables $x=s-{u},\,y={v}-s$, we have that:
\begin{align*}
\int_0^s e^{({u}-s)\theta}  \dif {u} \int_s^t e^{({v}-t)\theta}({v}-{u})^{2H-2}\dif {v}&=\int_0^s e^{-\theta x}  \dif x \int_0^{t-s} e^{(y-(t-s))\theta}(x+y)^{2H-2}\dif y\\
&\le C\big(1\wedge(t-s)^{2(H-1)}\big),
\end{align*}where in the last line we use Lemma~\ref{upper bound F-1}. Plugging the above two upper bounds into the inequality \eqref{qidan buds00}, we obtain the estimate \eqref{rho2 guji}.

Finally, we show the two estimates \eqref{rho21ts cha bds} or \eqref{rho21ts cha bds-new case} for the two different assumptions on $(G)$ respectively.
Case (i): When $(G)$ satisfies Hypothesis~ ($H_2'$)-($H_3$), it follows from Corollary~\ref{cor2-2} and Lemma~\ref{upper bound F} that
\begin{align}\label{zeta2 eta2 cha kongzhi}
\abs{\E[  \zeta_s^2]-\E[  \eta_s^2] }\le \left[\mu(e^{-\theta(s-\cdot)} \mathbf{1}_{[0,s]} (\cdot) )\right]^2\le C\times \left(1\wedge s^{2(H-1)} \right),
\end{align}
and 
\begin{align*}
&\abs{\rho_1(t,s)-\rho_2(t,s)}\notag \\
&\le e^{(s-t)\theta}\abs{\E[  \zeta_s^2]-\E[  \eta_s^2] }+ \int_0^s e^{ ({u}-s)\theta}  \dif  {u} \int_s^t e^{ ({v}-t)\theta} \abs{\frac{\partial ^2 R(u,v)}{\partial u\partial v}-\frac{\partial ^2 R^B(u,v)}{\partial u\partial v} }\dif  {v}\notag\\
&\le C\times \left[e^{(s-t)\theta} \left(1\wedge s^{2(H-1)} \right)+\int_0^s e^{({u}-s)\theta} u^{H-1}  \dif  {u} \int_s^t e^{ ({v}-t)\theta}  v^{H-1}\dif  {v} \right] \\
&\le   C\times  \left(1\wedge s^{2(H-1)} \wedge (t-s)^{H-1}\right), \notag
\end{align*} which is the desired \eqref{rho21ts cha bds}.

Case (ii). When $(G)$ is given as in Example~\ref{exmp lizi001}, it follows from the identities \eqref{guanjiandengshi 000-new1}-\eqref{innp fg3-zhicheng0-00} and \eqref{rho1ts bds}-\eqref{rho2ts bds} that $ \forall \, 0\le s\le t\le T$,
 \begin{align}\label{plug zhongjian}
0\le \rho_2(t,s)-\rho_1(t,s)= e^{(s-t)\theta}[\E(\eta_s^2-\zeta_s^2)].
\end{align} The identity \eqref{guanjiandengshi 000-new1} and Lemma ~\ref{upper bound F} imply that 
\begin{align*}
\E(\eta_s^2-\zeta_s^2)&= H\int_{0}^s e^{2(u-s)\theta} u^{2H-1}\dif u\le C \left(1\wedge s^{2H-1} \right) .
\end{align*}
Plugging the inequality displayed above into \eqref{plug zhongjian}, we obtain the estimate \eqref{rho21ts cha bds-new case}.
\end{proof}

Next, recall the chaos process $W_T$, which is represented by the identity  \eqref{wt dyi}. It follows from the product formula of Wiener-It\^o multiple integrals that
\begin{align}
\E[W_T^2]&=\frac{2}{  {T}}\int_{[0,T]^2}\rho_1^2(t,s)\dif t\dif s,\label{wt2 bds}\\
\kappa_3(W_T)&=\E[W_T^3]=\frac{48}{T^{\frac32}} \int_{0\le r<s<t\le T}\rho_1(r,s)\rho_1(s,t)\rho_1(t,r)\dif r\dif s\dif t,\label{wt2 bds2}\\
\kappa_4(W_T)&=\E[W_T^4]-3\big[\E[W_T^2]\big]^2\notag\\
&\le \frac{48\times 24}{T^{2}} \int_{0\le s<t<u<v\le T}\rho_1(s,t)\rho_1(t,u)\rho_1(u,v)\rho_1(v,s)\dif s\dif t\dif u\dif v.\label{wt2 bds3}
\end{align} 
For further details, please refer to   \cite{BBES 23, DEKN 22}.

{\bf Proof of Theorem~\ref{coro key guanj}:}
The proof is a direct application of the optimal fourth moment theorem, \cite{NP 15}. By using  Propositions~\ref{sijiejun juli qidian} and \ref{sijiejun juli} and the fourth moment theorem, \cite{NP 05}, it can be shown that when $H\in (0,\frac12)$, 
\begin{align*}
W_T  \stackrel{ {law}}{\to}  Z,\quad \text{ as\,\, } T\to\infty.
\end{align*}
Then, the optimal fourth moment theorem implies that when $H\in (0,\frac12)$, 
\begin{align*}
d_{TV}(W_T,\,Z)&\le C\times \max\set{\abs{\kappa_3(W_T),\,\kappa_4(W_T)}}\le C\times \frac{1}{\sqrt{T}},
\end{align*}
where the last inequality is from Propositions~\ref{sanjiejun juli} and \ref{sijiejun juli}.
 {\hfill\large{$\Box$}} 

{\bf Proof of Theorem~\ref{B-E bound thm}:}
Let $a= H\Gamma(2H)\theta^{-2H}  $ and consider $Z\sim N(0,\sigma^2_B)$. Firstly, for the moment estimator $\tilde{\theta}_T$, we can deduce from Theorem 1.3 and Remark 3.2 of \cite{chengu2021} that when the fractional Gaussian process $(G)$ satisfies Hypothesis~ ($H_2'$)-($H_3$) with $H\in (0,\,\frac12)$, the following inequality holds:
\begin{align*}
&\sup_{z\in \Rnum}\abs{P(\sqrt{\frac{4H^2 T}{\theta \sigma^2_{H}}} (\tilde{\theta}_T-\theta )\le z)-\Phi( z)}\\
&\le  C\times \left[ d_{Kol}(W_T, Z) +\sqrt {T}\abs{\frac{1}{  {T}}\int_{0}^T  \E[\zeta_t^2] \dif t - a}+\frac{1}{\sqrt{T}}\right]\\
&\le C\times \left[d_{Kol}(W_T, Z) + \frac{1}{\sqrt{T}}\right],
\end{align*}
which, together with Theorem~\ref{coro key guanj}, implies that the Berry-Ess\'{e}en upper bound \eqref{b-e bound 44}  holds. 

Secondly, for the informal LSE $\hat{\theta}_T$, we can deduce from the proof of Theorem 3.2 of \cite{BBES 23} that the following inequality holds:
\begin{align*}
& \sup_{z\in \Rnum}\abs{P(\sqrt{\frac{T}{\theta \sigma^2_{H}}} (\hat{\theta}_T-\theta )\le z)-\Phi( z)}\\
&\le C\times \left[ d_{Kol} (W_T,\, Z ) +\abs{\E[W_T^2]-\sigma^2_B}+\frac{1 }{\sqrt{T} }\right],
\end{align*}
which, together with Theorem~\ref{coro key guanj} and Proposition~\ref{sijiejun juli qidian}, implies that the Berry-Ess\'{e}en upper bound \eqref{b-e bound 34}  holds. 
{\hfill\large{$\Box$}} 


\section{Examples}\label{exmp section}
Recall that Example~\ref{exmp lizi001} dose not satisfy Hypothesis~($H_2'$)-($H_3$). In Subsection~\ref{exmp section-1}, we point out that most of other examples of fractional Gaussian processes in the literature belong to a sub-class of Hypothesis~($H_2'$)-($H_3$).  In Subsection~\ref{exmp section-2}, we will construct an artificial fractional Gaussian processes that satisfies Hypothesis~($H_2'$)-($H_3$) but does not belong to the sub-class.
\subsection{Gaussian processes {belonging} to a sub-class of Hypothesis~($H_2'$)-($H_3$)}\label{exmp section-1}
In \cite{chenlu2023}, a sub-class of Hypothesis~($H_2'$)-($H_3$) is introduced. Here, we state it as a single Hypothesis as follows:
\begin{hyp}\label{hypthe1new}
The covariance function $R(s,t)=E (G_tG_s)$ of the Gaussian process $(G_t)_{t\in[0,T ]}$ with zero mean satisfies Hypothesis~($H_1$)-($H_2'$) and 
\begin{enumerate}
 \item[($H_3'$)]There exist constants $H'>0,\, K\in (0,2)$ independent of $T$ such that $H:=H'K\in (0, 1)$, and there exist constants $C_1,C_2\ge 0$ which depend only on $H',\,K$ such that the inequality
  \begin{equation}\label{phi2}
 \left| \frac{\partial}{\partial s}\left(\frac{\partial R(s,t)}{\partial t} - \frac{\partial R^{B}(s,t)}{\partial t}\right)\right| \le  C_1  (t+s)^{2H-2}+C_2 (s^{2H'}+t^{2H'})^{K-2}(st)^{2H'-1} 
\end{equation}holds.
\end{enumerate}
\end{hyp}
\begin{remark}\label{zhushi 001}
\begin{enumerate}
 \item [(i)] Hypothesis ($H_3'$)  is slightly stronger than Hypothesis ($H_3$). In fact, the inequality $ a^2+b^2\ge 2ab$ and the monotonicity of the function $f(x)=x^{p},\,x>0$ with $p<0$ imply that the inequality
\begin{align*} 
C_1  (t+s)^{2H-2}+C_2 (s^{2H'}+ t^{2H'})^{K-2}(st)^{2H'-1}& \le  \left(2^{2H-2}C_1 +2^{K-2}C_2 \right) (st )^{H-1},
\end{align*}
holds when $K\in (0,2),\,H:=H'K\in (0,1)$. Hence, Hypothesis ($H_3'$)  implies that Hypothesis ($H_3$)  holds, i.e.,
\begin{align}\label{bijiao 00 old}
\left|\frac{\partial}{\partial s}\left(\frac {\partial R(s,t)}{\partial t} - \frac {\partial R^{B}(s,t)}{\partial t}\right)\right|  & \le  C (ts )^{H-1}.
\end{align} 
Thus, Hypothesis~\ref{hypthe1new} is a sub-class of Hypothesis~($H_2'$)-($H_3$).
\item[(ii)] One special case of Hypothesis~($H_3'$) is as follows:
   there exist two constants ~$H\in (0, 1)$ and $C_H\ge 0$ {independent} of $T$ such
   that the inequality
\begin{equation}\label{phi2-dj}
 \left| \frac{\partial}{\partial s}\left(\frac {\partial R(s,t)}{\partial t} - \frac {\partial R^{B}(s,t)}{\partial t}\right)\right| \le C_H(t+s)^{2H-2} 
  \end{equation}
  holds. 
   Another special case of Hypothesis ~($H_3'$)  is as follows:
   there exist three constants $H'>0,\, K\in (0,2)$ and $C_{H',K}\ge 0$ that do not depend on $T$, such that $H:=H'K\in (0, 1)$ and the inequality
  \begin{equation}\label{phi2-dj-1}
 \left| \frac{\partial}{\partial s}\left(\frac {\partial R(s,t)}{\partial t} - \frac {\partial R^{B}(s,t)}{\partial t}\right)\right| \le  C_{H',K}   (s^{2H'}+t^{2H'})^{K-2}(st)^{2H'-1} 
  \end{equation}
  holds. 
 \end{enumerate}
\end{remark}	

For the next six types of Gaussian processes,  the inequalities \eqref{phi2-dj}, \eqref{phi2-dj-1} and \eqref{phi2} degenerate respectively and hence Hypothesis~\ref{hypthe1new} holds. 
\begin{ex}\label{exmp6}
The Gaussian process $\vartheta_t$ with the covariance function
$$ R(t,\, s)=\frac{1}{2}\left((s+t)^{2H}-|t-s|^{2H}\right) $$ 
satisfies Hypothesis~\ref{hypthe1new} when $H \in (0,1) $. 
This Gaussian process is the derivative of the negative sub-fractional Brownian motion with parameter $h=2+2H$. See \citealt{BGT 2007}. In fact, it is clear that 
\begin{align*}
   \frac{\partial}{\partial s}\left(\frac{\partial R(s,t)}{\partial t} - \frac{\partial R^{B}(s,t)}{\partial t}\right)=H(2H-1)(t+s)^{2H-2},   
\end{align*}
which makes the inequality \eqref{phi2-dj} degenerate.
\end{ex}
\begin{example}\label{exmp0001-5}
The sub-fractional Brownian motion $\{S^H(t), t \geq 0\}$ with parameter $H\in (0,1)$ and the covariance function
$$R(t,s)=s^{2H}+t^{2H}-\frac{1}{2}\left((s+t)^{2H}+|t-s|^{2H}\right)$$
satisfies Hypothesis~\ref{hypthe1new}.
 \end{example}
\begin{example} The bi-fractional Brownian motion $\{B^{H',K}(t), t\geq 0\}$ with the covariance function
$$R(t,s)=\frac{1}{2}\left((s^{2H'}+t^{2H'})^K - |t-s|^{2H'K}\right) $$ satisfies Hypothesis~\ref{hypthe1new}, where parameters $H'>0,K \in (0, 2)$ such that $H'K\in (0, 1)$. 
 \end{example}
\begin{example}\label{exmp5}
The generalized sub-fractional Brownian motion $S^{H',K}(t) $ with parameters $H' \in (0, 1),\,K \in(0,2)$ and $H'K\in (0,1)$ satisfies Hypothesis~\ref{hypthe1new}. The covariance function
is given by
$$ R(t,\, s)= (s^{2H'}+t^{2H'})^{K}-\frac12 \left[(t+s)^{2H'K} + \abs{t-s}^{2H'K} \right].$$
\end{example}
\begin{ex}\label{exmp7}  
The self-similar Gaussian process $\mathsf{S}^{H,K}(t) $ with parameters $H\in (0,\frac12),\,K\in (0, 1)$ and the covariance function
$$ R(t,\, s)=\frac{1}{2(1-K)}[t^{2H}+s^{2H}-K(t+s)^{2H}]-\frac12 |t-s|^{2H}$$  satisfies Hypothesis~\ref{hypthe1new}, since it also makes the inequality \eqref{phi2-dj} degenerate. 
See \citealt{Sgh 2014}. \end{ex}

\begin{ex}\label{exmp7-1zili}
The generalized fractional Brownian motion is an extension of both fractional and sub-fractional Brownian motions, with a covariance function given by
$$ R(t,\, s)=\frac{(a+b)^2}{2(a^2+b^2)}(s^{2H}+t^{2H})-\frac{ab}{a^2+b^2}(s+t)^{2H}-\frac12 \abs{t-s}^{2H},$$
where $H\in (0, 1)$ and $(a,b)\neq (0,0)$. It satisfies Hypothesis~\ref{hypthe1new} since it also makes the inequality \eqref{phi2-dj} degenerate.  See \citealt{Zili 17}. In fact, both Example~\ref{exmp6} and Example~\ref{exmp7} are special cases of this one.
\end{ex}

\subsection{An artificial counterexample that dose not satisfy Hypothesis~\ref{hypthe1new}, but still satisfies Hypothesis~($H_2'$)-($H_3$).}\label{exmp section-2}
The following example is a mixed Gaussian process, which is a linear combination of independent centred Gaussian processes.
\begin{ex}\label{exmp8-counter}
Suppose that $Z\sim N(0,1)$ is independent on the fBm $(B^H_t)_{t\in[0,T ]}$. Construct a Gaussian process as follows:
\begin{align*}
G_t:=B^H_t+t^H Z,\qquad t\in [0,T ],
\end{align*} 
where the covariance function satisfies
\begin{align*}
R(t,s)-R^{B}(t,s) =(ts)^H.
\end{align*} 
\end{ex}

\section{Discussion: a variant of Hypothesis~\ref{hypthe1}.}\label{sec-dis}
In this section, we always assume that the following Hypothesis~\ref{hypthe6} holds.
 \begin{hyp}\label{hypthe6}
 The covariance function $R(s,t)=E (G_tG_s)$ of the centred Gaussian process $(G_t)_{t\in[0,T ]}$ with $G(0)=0$ satisfies ($H_1$) and 
\begin{enumerate}
  \item[($H_4$)] for any fixed $t \in[0,T]$,  the first-order partial derivative
$$\frac {\partial }{ \partial t}  R(s,t)$$ 
is a normalized bounded variation function on $s\in  [0,T]$.
  \end{enumerate}
\end{hyp}
To compare with Hypothesis~\ref{hypthe1},  the noticeable difference between Hypothesis~($H_2$) and Hypothesis~($H_4$) is that the Gaussian process $(G) $ satisfying Hypothesis~\ref{hypthe6} is no longer related to the fBm. Thus, Theorem~\ref{prop 2-1} is not valid, although Proposition~\ref{prop 2-1010} still holds. However, we can use  a similar idea to modify Theorem~\ref{prop 2-1}. For example, the following four types of Gaussian processes in the literature satisfy Hypothesis~\ref{hypthe6}. We will study the similar problems in Section~\ref{sec3 applications} for these types of Gaussian processes in a separate paper. Here we only list a few obvious results without proofs. 

Sometimes, Hypothesis ($H_{4}$) is taken as the following special case: 
  \begin{hyp}\label{hypthe2new-new}
The covariance function $R(s,t)=E (G_tG_s)$ of the Gaussian process $(G_t)_{t\in[0,T ]}$ with zero mean satisfies Hypothesis~($H_1$) and \begin{enumerate}
  \item[($H_4'$)] for any fixed $t \in[0,T]$, the partial derivative 
  $\frac {\partial R(s,t)}{\partial t} $
  is an absolutely continuous function on $s\in [0,T]$.
\end{enumerate}
\end{hyp}
\begin{proposition}\label{prop-4-11}
If the covariance function $R(t,s)$ satisfies Hypothesis~\ref{hypthe2new-new}, then
  \begin{align} \label{diyierdengshi 00}
 \langle f,\,g \rangle_{\FH} &=\int_{[0,T]^2}  {f}(t)g(s)  \frac {\partial^2 R(s,t)}{\partial s\partial t}    \dif t \dif s,\qquad \forall f,\, g\in \mathcal{V}_{[0,T]}.
 \end{align} 
In addition, if the following Hypothesis~($H_5$) is also satisfied:
\begin{enumerate}     \item[($H_5$)] There exist two positive constants $C$ and $H$, independent of $T$, such that
  \begin{align}\label{bijiao 00 new}
\left| \frac{\partial^2 R(s,t)}{\partial s\partial t} \right| & \le  C(ts )^{H-1},\quad a.s.
\end{align} 
holds.
\end{enumerate}
Then, we can derive the following inequality:
\begin{equation}\label{inequality 210}
\abs{\langle f,\,g \rangle_{\mathfrak{H}}  }\leq  C \mu(\abs{f})\mu(\abs{g}),\qquad \forall f,\, g\in \mathcal{V}_{[0,T]}.
\end{equation} 
\end{proposition}

Proposition~\ref{prop-4-11} can be {derived by making a slight modification to} Corollary~\ref{cor2-2}.
We can easy verify that the following two examples satisfy Hypothesis~\ref{hypthe2new-new} and ($H_5$).
\begin{example}
Suppose that $H \in (0, \frac12)\cup(\frac12,1) $, and the constant $C_H$ is given by
\begin{align*}
    C_H=\left\{
      \begin{array}{ll}
\frac{H}{\Gamma(1-2H)}, & \quad  H\in (0,\frac12),\\
\frac{H(2H-1)} {\Gamma(2-2H)}, &\quad H\in (\frac12,1),
 \end{array}
\right.
\end{align*} and $\{W_t,t \ge 0 \}$ is the standard Brownian motion. 
The covariance function of the following Gaussian process
$$X_{t}  = \sqrt{C_H}\int_{0}^{\infty} \left( 1-e^{-r t} \right) r^{-\frac {1+2H}{2}} \dif W_{r} ,\quad t\ge 0$$ is given by
\begin{equation*}
R(s,t)=\left\{
      \begin{array}{ll}
\frac12\left[t^{2H}+s^{2H} -(t+s)^{2H}\right], & \quad  H\in\left(0,\frac12\right),\\
\frac12\left[(t+s)^{2H}-t^{2H}-s^{2H}\right], &\quad H\in\left(\frac12,1\right).
 \end{array}
\right.
\end{equation*}
Please refer to \citealt{Bardina2009, Lei2009}. 
\end{example}

\begin{example}
The covariance function of the trifractional Brownian motion $Z^{H',K}(t) $ with parameters $H' \in (0, 1),\,K \in(0,1)$ 
is given by
$$ R(t,\, s)=t^{2H'K} + s^{2H'K}- (t^{2H'}+s^{2H'})^{K}  .$$
Please refer to  \citealt{Lei2009, ma2013}. The Gladyshev's Theorem for this trifractional Brownian motion is obtained in \cite{Han 21}.
\end{example}
 
The following example does not satisfy ($H_5$), but still satisfies Hypothesis~\ref{hypthe2new-new}.
\begin{example}
Denote $\beta(\cdot,\cdot)$ as the beta function. Suppose that $a>-1, 0<{b}<1\wedge ( 1+a)$ and $2H=a+b+1$. The covariance function of the weighted-fractional Brownian motion $\{B_t^{a,b}, t \geq 0\}$ is given by
\begin{align*}
R(s,t)&=\frac{1}{2\beta(a+1,b+1)}\int_{0}^{s\wedge t}u^a \left[ (t-u)^b +(s-u)^b  \right] \dif u\\ 
&=\frac12\Big[t^{2H}+s^{2H}-\frac{1}{\beta(a+1,b+1)}\int_{s\wedge t}^{s\vee t}u^a (t \vee s -u)^b \dif u\Big].
\end{align*}
Please refer to \citealt{AAE 21, BGT 2007}. By Propositions~\ref{prop-4-11}, we have that if $f,\, g\in \mathcal{V}_{[0,T]}$, then
\begin{align}
 \langle f,\,g \rangle_{\FH} &= \int_{[0,T]^2}f(t)g(s) \frac{\partial^2 R(s,t)}{\partial s \partial t } \dif t\dif s,
\end{align}
where 
\begin{align*}
 \frac{\partial^2 R(s,t)}{\partial s \partial t }=\frac{b}{\beta(a+1,b+1)} (t\wedge s)^{a} \abs{t-s}^{b-1}.
\end{align*}
\end{example}
Finally, we present a Gaussian process from the literature in which the first-order partial derivative $\frac {\partial }{ \partial t}  R(s,t)$ is not an absolutely continuous function of $s\in [0,T]$.
\begin{example} Suppose that $H \in (0, \frac12)$. The Gaussian process $\{G_t , t\ge 0\} $ has the covariance function 
$$ R(s, t)= \frac12\big[(t+s)^{2H}-(\max(s,t))^{2H} \big] . $$
Please refer to Theorem 1.1 of \cite{Talarczyk2020}.  Clearly,  for any fixed $t\in [0,T]$, the first-order partial derivative\begin{equation*}
\frac {\partial }{ \partial t}  R(s,t)=\left\{
      \begin{array}{ll}
H(t+s)^{2H-1}-H t^{2H-1}, & \quad 0<s\le t,\\
H(t+s)^{2H-1}, & \quad t<s\le T
 \end{array}
\right. 
\end{equation*}
is a bounded variation function of $s\in [0,T]$. It is a linear combination of an absolutely continuous function 
\begin{equation*}
\varphi(s)=H(t+s)^{2H-1}-H t^{2H-1},\quad s\in[0,T],
\end{equation*}and a step function 
\begin{equation*}
\phi(s)=\left\{
      \begin{array}{ll}
0, & \quad 0<s\le t,\\
 H t^{2H-1}, & \quad t<s\le T.
 \end{array}
\right. 
\end{equation*}
Similar to the identity \eqref{guanjiandengshi 000-new1}, we have that if $f,\, g\in \mathcal{V}_{[0,T]}$ and at least one of them is continuous, then
 \begin{align}\label{guanjiandengshi 000-new2}
 \langle f,\,g \rangle_{\FH} &=H(2H-1)\int_{[0,T]^2}f(t)g(s) (t+s)^{2H-2}\dif t\dif s +H \int_0^T  {f}(t) g(t)t^{2H-1}\dif t .
 \end{align}
It is clear in this case the limit \eqref{h jixian} is no longer valid, and hence the conditions of \cite{Nor 11} do not hold. Finally, we point out that if we interpret the mixed second-order partial derivative $$ \frac{\partial }{\partial s}\left(  \frac{\partial R(s,t)}{\partial t }\right)$$ as the distributional derivative of $ \frac{\partial R(s,t)}{  \partial t }$ with respective to the variable $s$, then the identity \eqref{guanjiandengshi 000-new2} can be seen as a special case of the identity \eqref{diyierdengshi 00}. The same viewpoint can also be used to illustrate the identities \eqref{guanjiandengshi 000-new1} and \eqref{diyierdengshi}. 
\end{example}
 
\section{Appendix}\label{appendx}
The following lemma is trivial, the reader can refer to \cite{Chenzhou2021, chengu2021}. 
\begin{lemma} \label{upper bound F}
Assume that $\beta>0$ and $\theta>0$. Denote $$ \bar{A}(s)=e^{-\theta s}\int_0^{s} e^{\theta r} r^{\beta -1}\dif r. $$ Then 
there exists a positive constant $C$ such that for any  $s\in [0,\infty)$,
\begin{align*}
\bar{A}(s)&\le C \times\left(s^{\beta}\mathbbm{1}_{[0,1]}(s) + s^{\beta-1}\mathbbm{1}_{ (1,\,\infty)}(s)\right)\le  C \times \left (s^{\beta-1} \wedge s^{\beta}\right).
\end{align*}
Especially, when $\beta\in (0,1)$, there exists a positive constant $C$ such that for any  $s\in [0,\infty)$,
\begin{align*}
\bar{A}(s)&\le C \times(1\wedge s^{\beta-1}).
\end{align*}
\end{lemma} 
\begin{lemma}\label{upper bound F-1}
Suppose that $H\in (0,1)$.
There exists a positive constant $C$ such that for all $t> 0$,
\begin{align*}
\int_0^{\infty}e^{-x} \dif x\int_0^t e^{y-t} (x+y)^{2H-2} \dif y\le C(1\wedge t^{2H-2}).
\end{align*}
\end{lemma}
\begin{proof}
First, by the inequality $x+y\ge 2\sqrt{xy}$ and Lemma~\ref{upper bound F}, the above integral is uniformly bounded for all $t> 0$. 

Second, it follows from L'H\^opital's rule and the monotone convergence theorem that 
\begin{align*}
\lim_{t\to \infty} \frac{\int_0^t e^{y} \dif y\int_0^{\infty}e^{-x}(x+y)^{2H-2}  \dif x }{e^t t^{2H-2}}&=\lim_{t\to \infty} \frac{e^t \int_0^{\infty} e^{-x}(x+t)^{2H-2}  \dif x } {e^t t^{2H-2}}\\
&=\lim_{t\to \infty}\int_0^{\infty} e^{-x} (1+\frac{x}{t})^{2H-2}\dif x = \int_0^{\infty} e^{-x} \dif x =1,
\end{align*}
which implies that there exists a positive constant $C$ such that for all $t> 0$,
\begin{align*}
\int_0^{\infty}e^{-x} \dif x\int_0^t e^{y-t} (x+y)^{2H-2} \dif y\le C  t^{2H-2}.
\end{align*}
\end{proof}

\subsection{Integration by parts of signed Lebesgue-Stieltjes measure}
Denote $\mu_F$ the signed Lebesgue-Stieltjes measure of the bounded variation function $F$. The following integration by parts formula is taken from Exercise 3.34 (b) of \cite{Foll 99}.
\begin{lemma}
Suppose that $-\infty<a<b<\infty$ and $F,G$ are normalized bound variation on $\R$. If there are no points in $[a,b]$ where $F$ and $G$ are both discontinuous, then 
\begin{align}\label{int by parts Folland 00}
 \int_{[a,b]} G \mu_F(\dif x)+\int_{[a,b]} F\mu_G(\dif x)=F(b)G(b)-F(a-)G(a-).
\end{align}
If, in addition, $F$ is absolutely continuous on $[a,b]$, then 
\begin{align}\label{int by parts Folland}
\int_{[a,b]} G F'\dif x + \int_{[a,b]} F\mu_G(\dif x)=F(b)G(b)-F(a-)G(a-).
\end{align}
\end{lemma} 
Please also refer to Exercise 1.7.13  of \cite{tao2011} for the formula \eqref{int by parts Folland}.

\begin{lemma}\label{partial integral}
Suppose that $[a,b]$ is a compact interval with positive length. Denote $\mathcal{V}_{[a,b]}$ as the set of bounded variation functions on $[a,b]$ and denote  $\nu_g$ as the restriction on $\left({[a,b]},\mathcal{B}({[a,b]})\right)$ of the signed \textnormal{Lebesgue-Stieljes} measure $\mu_{g^0}$ on $\left(\Rnum,\mathcal{B}(\Rnum)\right)$  of
\begin{equation*}
g^0(x)=\left\{
    \begin{array}{ll}
g(x), & \quad \text{if}~x\in [a,b],\\
0, &\quad \text{otherwise}.
 \end{array}
\right.
\end{equation*}
where $g\in\mathcal{V}_{[a,b]}$. 
\begin{enumerate}  \item[(a)]
If $f: [a,b]\rightarrow\R$ is absolutely continuous on~$[a,b]$ and $g\in\mathcal{V}_{[a,b]}$,  then we have
\begin{align}
  -\int_{[a,b]}g(t) f'(t)\dif t=\int_{[a,b]}f(t) {\nu_g}(\dif t). \label{01}
\end{align}
 \item[(b)] If 
$f,\,g\in\mathcal{V}_{[a,b]}$ and if there are no points in $[a,b]$ where $f$ and $g$ are both discontinuous, then we have
 \begin{align}
	-\int_{[a,b]}g(t)  \mu_f(\dif t)=\int_{[a,b]}f(t) {\nu_g}(\dif t). \label{01001}
\end{align}
 \end{enumerate} 
\end{lemma}
\begin{remark}
\begin{enumerate}
  \item[(1)]   The measure $\nu_g$ or $\mu_{g^0}$ may have atoms, i.e.,
  \begin{align*}
 \nu_g(\{a\})&=g^0(a+)-g^0(a-)= g(a+)-0=g(a+),\\
   \nu_g(\{b\})&=g^0(b+)-g^0(b-)=0-g(b-)=-g(b-).
 \end{align*} 
  \item[(2)]  Another method to derive the integration by parts formula \eqref{01} builds
upon \cite{Jolis 2007}. For further details, please refer to \cite{withdingzhen}. 
   \item[(3)] If $f: [a,b]\rightarrow\R$  is a piecewise constant jump function, then the Lebesgue-Stieltjes measure $\mu_f$ can be represented as
\begin{align*}
\mu_{f}=\sum_{n} c_n\delta_{x_n}(\cdot),
\end{align*} where $\delta_x(\cdot)$ is the Dirac measure at $x$. The formula \eqref{01001} can be rewritten as:
\begin{align}\label{01001-new-00}
	-\sum_n c_n g(x_n) =\int_{[a,b]}f(t) {\nu_g}(\dif t).
\end{align}
\end{enumerate}
 \end{remark}
 
\begin{proof}
We decompose $g^0(x)$ into the difference of two normalized bounded variation functions:
\begin{align}\label{basic decomp}
g^0(x)=g_1^0(x)-g_2^0(x),
\end{align}
where 
\begin{equation*}
g_1^0(x)=\left\{
      \begin{array}{ll}
      0,&\quad \text{if } x<a,\\
 g(x), & \quad \text{if } x\in [a,b],\\
g(b), &\quad \text{if } x>b,    
 \end{array}
\right.
\end{equation*}
and 
\begin{equation*}
g_2^0(x)=\left\{
      \begin{array}{ll}
      0,&\quad \text{if } x<b,\\
g(b), &\quad \text{if } x>b.    
 \end{array}
\right.
\end{equation*} 
Recall the two identities \eqref{ceducha}-\eqref{g02 cedu}: 
\begin{equation*}
\mu_{g^0} =\mu_{g^0_1}-\mu_{g^0_2},\qquad
\mu_{g^0_2}=g(T) \delta_T(\cdot).
\end{equation*}
(a): If $f$ is absolutely continuous on $[a,b]$, then the integration by parts formula \eqref{int by parts Folland} implies that 
\begin{align}\label{int by parts Folland 000}
\int_{[a,b]} g f'\dif x + \int_{[a,b]} f  {\mu_{g^0_1} }(\dif x)=g(b)f(b)=\int_{[a,b]} f  {\mu_{g^0_2} }(\dif x),
\end{align}
which further implies that  if $f$ is absolutely continuous on $[a,b]$, then 
\begin{align}
-\int_{[a,b]} g f'\dif x =\int_{[a,b]} f  {\mu_{g^0} }(\dif x).
\end{align}

(b): 
Since there are no points in $[a,b]$ where $f$ and $g$ are both discontinuous, the identity \eqref{int by parts Folland 00} implies that 
\begin{align*}
-\int_{[a,b]} g \mu_f(\dif x) =\int_{[a,b]} f  {\mu_{g^0_1} }(\dif x)-\int_{[a,b]} f  {\mu_{g^0_2} }(\dif x) =\int_{[a,b]} f  {\mu_{g^0} }(\dif x).
\end{align*}
\end{proof}

 

\subsection{Calculating the second, third and fourth cumulants of $W_T$}
 In this subsection, we always assume that either the fractional Gaussian process $(G)$ satisfies Hypothesis~ ($H_2'$)-($H_3$) or is given as in Example~\ref{exmp lizi001}.

\begin{notation}
We denote a second Wiener chaos with respect to the fBm $(B^H)$ as follows:
\begin{align}
D_T-\E D_T&=\frac{1}{\sqrt {T}}\int_{0}^T \left(\eta_t^2 -\E[\eta_t^2] \right)\dif t.
\end{align}
\end{notation}
\begin{proposition}\label{sijiejun juli qidian} 
Let $ W_T$ and $\sigma^2_B$ be given as in \eqref{wt dyi} and \eqref{sb2 dyi} respectively. Then we have 
 \begin{align}
\abs{\E[W_T^2]-\sigma^2_B}\le \frac{C}{\sqrt{T}} \quad \text{or}\quad  \frac{C}{ {T}^{ (1-2H)\wedge\frac12} },
\end{align}
when $(G)$ satisfies Hypothesis~ ($H_2'$)-($H_3$) or  is given as {described} in Example~\ref{exmp lizi001} respectively.
\end{proposition}
\begin{proof}
First, recall that \cite{BBES 23}
\begin{align*}
\abs{\E\Big[\big(D_T-\E D_T\big) ^2\Big]-\sigma^2_B}\le \frac{C}{\sqrt{T}}.
\end{align*}
It follows from the identity \eqref{wt2 bds} that
\begin{align*} 
 { \E[W_T^2]-\E[\left(D_T-\E D_T\right) ^2] }=\frac{4}{{T}}{\int_{ 0\le s<t\le T} \left(\rho_1^2(t,s)- \rho_2^2(t,s)\right)\dif t\dif s}. 
\end{align*}
Hence, it is clear that we only need to show the following upper bounds in the two different cases:
\begin{align}\label{cha rho120}
 \abs{\int_{ 0\le s<t\le T} \big(\rho_1^2(t,s)- \rho_2^2(t,s)\big)\dif t\dif s} \le  {C}T^H  \text{\,\, or\,\, } {C}{ {T}^{2H} }.
\end{align}

Case 1: When $(G)$ satisfies Hypothesis~ ($H_2'$)-($H_3$) and $H\in (0,\frac12)$.
It follows from Proposition~\ref{key prop bijiao} that  for any $0\le s<t$,
\begin{align*}
\abs{\rho_1 (t,s)- \rho_2 (t,s)}\le C\left(1\wedge s^{2(H-1)} \right),\quad \abs{\rho_1 (t,s)+\rho_2 (t,s)}\le C \left(1\wedge (t-s)^{H-1}\right), 
\end{align*}
This implies that:
\begin{align*}
&\abs{\int_{ 0\le s<t\le T} \left(\rho_1 (t,s)- \rho_2 (t,s)\right) \left(\rho_1 (t,s)+\rho_2 (t,s)\right) \dif t\dif s}\notag\\
&\le C\int_{ 0\le s<t\le T} \left(1\wedge s^{2(H-1)}\right) \left(1\wedge (t-s)^{H-1}\right)\dif t\dif s\notag\\
&<C \int_0^{\infty} \left(1\wedge s^{2(H-1)} \right)\dif s\int_s^{T} \left(1\wedge (t-s)^{H-1}\right)\dif t<CT^H\label{diyibufen}
\end{align*}
since $H\in (0,\frac12)$.

Case 2: When $(G)$ is given as in Example~\ref{exmp lizi001}. Proposition~\ref{key prop bijiao} implies that for all $0\le s<t$,
\begin{align*}
\abs{\rho_1 (t,s)+\rho_2 (t,s)}\le 2  \abs{\rho_2 (t,s)} + \abs{\rho_2(t,s)-\rho_1(t,s)}\le C \left(1\wedge (t-s)^{2(H-1)}\right).
\end{align*}
Since $H\in (0,\frac12)$, we have
\begin{align*}
\abs{\int_{ 0\le s<t\le T} \left(\rho_1^2 (t,s)- \rho_2^2 (t,s)\right)\dif t\dif s} 
&\le C\int_{ 0\le s<t\le T} s^{2H-1} \left(1\wedge(t-s)^{2(H-1)}\right)\dif t\dif s \\
&\le C \int_0^{T}  s^{2H-1}  \dif s \le CT^{2H }.
\end{align*} 

\end{proof}

\begin{proposition}\label{sanjiejun juli}
Let $\kappa_3(W_T)$ be defined as in equation \eqref{wt2 bds2}. Then we have 
\begin{align}
\abs{\kappa_3(W_T)}\le \frac{C}{\sqrt{T}}.
\end{align}
\end{proposition}
\begin{proof}
First, recall that \cite{BBES 23}, we have
\begin{align*}
\abs{\kappa_3\left(D_T-\E D_T\right)}\le \frac{C}{\sqrt{T}}.
\end{align*}
Then, it follows from the identity \eqref{wt2 bds2} that we only need to show that 
\begin{align}\label{cha rho12}
\abs{ \int_{0\le r<s<t\le T}\left[\rho_1(r,s)\rho_1(s,t)\rho_1(t,r) - \rho_2(r,s)\rho_2(s,t)\rho_2(t,r)\right]\dif r\dif s\dif t}\le CT^{H}\, \text{or}\,\, CT^{2H},
\end{align}
when $(G)$ satisfies Hypothesis~ ($H_2'$)-($H_3$) or is given as in Example~\ref{exmp lizi001} respectively.

Case 1: Suppose $(G)$ satisfies Hypothesis~ ($H_2'$)-($H_3$). We divide the proof into three steps.
Step 1.  Since $H\in (0,\frac12)$, it follows from Proposition~\ref{key prop bijiao} that 
\begin{align}
& \abs{ \int_{0\le r<s<t\le T} \left(\rho_1(r,s)-\rho_2(r,s)\right) \rho_2(s,t)\rho_2(t,r) \dif r\dif s\dif t}\notag\\
 &<C \int_{0 }^{\infty} \left(1\wedge r^{2(H-1)}\right)\dif r\int_r^T \left(1\wedge (t-r)^{2(H-1)}\right)\dif t \int_r^t \left(1\wedge(t-s)^{2(H-1)}\right) \dif s\notag\\
 &\le \left(\int_{0 }^{\infty} \left(1\wedge r^{2(H-1)}\right) \dif r \right)^3\le C.\label{rho12cha 00001}
\end{align} 
Similarly, we have
\begin{align}
  \abs{ \int_{0\le r<s<t\le T} \rho_2( r,s)\left(\rho_1(s,t)-\rho_2(s,t)\right)\rho_2(t,r)\dif r\dif s\dif t} &\le C   ,\\
 \abs{\int_{0\le r<s<t\le T} \rho_2( r,s) \rho_2(s,t)\left(\rho_1(t,r)-\rho_2(t,r)\right) \dif r\dif s\dif t}&\le C  . \label{rho12cha 000013}
\end{align}
Step 2. It follows from Proposition~\ref{key prop bijiao} that 
\begin{align}
& \abs{ \int_{0\le r<s<t\le T} \left(\rho_1(r,s)-\rho_2( r,s)\right) \left(\rho_1(s,t)-\rho_2(s,t)\right)\rho_2(t,r)\dif r\dif s\dif t}\notag\\
&\le C  \int_{0\le r<s<t\le T} \left(1\wedge r^{2(H-1)} \right) \left(1\wedge s^{2(H-1)} \right)\left(1\wedge (t-r)^{2(H-1)} \right) \dif r\dif s\dif t\notag\\
&\le C \left(\int_{0 }^{\infty}  \left(1\wedge r^{2(H-1)} \right) \dif r  \right)^3
\le C .\label{r1r2chacha2 r2}
\end{align}
Similarly, we have 
\begin{align}
 \abs{\int_{0\le r<s<t\le T} \left(\rho_1(r,s)-\rho_2( r,s)\right)\rho_2(s,t) \left(\rho_1(t,r)-\rho_2(t,r)\right) \dif r\dif s\dif t}& \le C T^{H} ,\\
  \abs{\int_{0\le r<s<t\le T} \rho_2( r,s) \left(\rho_1(s,t)-\rho_2(s,t)\right)\left(\rho_1(t,r)-\rho_2(t,r)\right) \dif r\dif s\dif t} &\le  C T^{H} .\label{r1r2chacha2 r2-3}
\end{align}

Step 3. It follows from Proposition~\ref{key prop bijiao} that  \begin{align}
& \abs{ \int_{0\le r<s<t\le T} \left(\rho_1(r,s)-\rho_2( r,s)\right) \left(\rho_1(s,t)-\rho_2(s,t)\right) \left(\rho_1(t,r)-\rho_2(t,r)\right) \dif r\dif s\dif t}\notag\\
&\le C \int_{0\le r<s<t\le T} \left(1\wedge r^{2(H-1)} \right) \left(1\wedge s^{2(H-1)} \right)  \left(1\wedge (t-r)^{H-1}\right) \dif r\dif s\dif t 
\le CT^{H}.\label{r1r2chacha3 }
\end{align} 

Finally, by plugging the seven upper bounds \eqref{rho12cha 00001}-\eqref{r1r2chacha3 } into the left hand side of \eqref{cha rho12}, we obtain the inequality \eqref{cha rho12}.

Case 2: Let $(G)$ be given as described in Example~\ref{exmp lizi001}.  By applying the inequality \eqref{rho21ts cha bds-new case}, we can replace all seven upper bounds in \eqref{rho12cha 00001}-\eqref{r1r2chacha3 } of Step 1 with $CT^{2H}$. 
Therefore, the second case of the inequality \eqref{cha rho12} holds.
\end{proof}

\begin{proposition}\label{sijiejun juli}
Let $\kappa_4(W_T)$ be defined as in \eqref{wt2 bds3}. Then we have 
\begin{align}
\abs{\kappa_4(W_T)}\le \frac{C}{\sqrt{T}}.
\end{align} 
\end{proposition}
\begin{proof}
 It follows from \cite{BBES 23} that 
\begin{align*}
\abs{\kappa_4\left(D_T-\E D_T\right)}\le \frac{C}{\sqrt{T}},
\end{align*}
which, together with the identity \eqref{wt2 bds3}, implies that we only need to show that 
\begin{align}
&\abs{\int_{0\le s<t<u<v\le T}\left(\rho_1(s,t)\rho_1(t,u)\rho_1(u,v)\rho_1(v,s)-\rho_2(s,t)\rho_2(t,u)\rho_2(u,v)\rho_2(v,s)\right)\dif s\dif t\dif u\dif v}\notag\\
&\le C T^H \quad \text{or}\quad CT^{2H},\label{kappa 4 mubiao}
\end{align}
when $(G)$ satisfies Hypothesis~ ($H_2'$)-($H_3$) or is given as in Example~\ref{exmp lizi001} respectively.

Case 1: When $(G)$ satisfies Hypothesis~ ($H_2'$)-($H_3$). We divide it's proof into four steps.
Step 1.  It follows from Proposition~\ref{key prop bijiao} that 
\begin{align}
& \abs{ \int_{0\le s<t<u<v \le T} \left(\rho_1(s,t)-\rho_2(s,t)\right) \rho_2(t,u)\rho_2(u,v)\rho_2(v,s)  \dif s\dif t\dif u \dif v}\notag\\
&<C \int_{0\le s<t<u<v \le T}\left(1\wedge s^{2(H-1)}\right) \left(1\wedge ( u-t)^{2(H-1)} \right)\left(1\wedge(v-u)^{2(H-1)}\right) \notag\\
 &\times\left(1\wedge(v-s)^{2(H-1)}\right) \dif s\dif t\dif u \dif v \notag\\
 &<C\left(\int_{0 }^{\infty} \left(1\wedge s^{2(H-1)}\right) \dif s\right)^4 
<C .\label{kappa4 di1}
\end{align} 
Similarly, we have
\begin{align}
  \abs{ \int_{0\le s<t<u<v \le T}\rho_2(s,t) \big(\rho_1(t,u)- \rho_2(t,u)\big)\rho_2(u,v)\rho_2(v,s)  \dif s\dif t\dif u \dif v}&<C,\\
   \abs{ \int_{0\le s<t<u<v \le T}\rho_2(s,t)\rho_2(t,u) \big(\rho_1(u,v)- \rho_2(u,v)\big)\rho_2(v,s)  \dif s\dif t\dif u \dif v}&<C,\\
   \abs{ \int_{0\le s<t<u<v \le T}\rho_2(s,t)\rho_2(t,u) \rho_2(u,v)\big(\rho_1(v,s)- \rho_2(v,s) \big) \dif s\dif t\dif u \dif v}&<C. \label{kappa4 di1-4}
   \end{align}
Step 2.  It follows from Proposition~\ref{key prop bijiao} that 
\begin{align}
& \abs{\int_{0\le s<t<u<v \le T} \left(\rho_1(s,t)-\rho_2(s,t)\right) \left(\rho_1(t,u)-\rho_2(t,u)\right)\rho_2(u,v)\rho_2(v,s)  \dif s\dif t\dif u \dif v}\notag\\
 &\le C \int_{0\le s<t<u<v \le T} \left(1\wedge s^{2(H-1)}\right)    \left(1\wedge t^{2(H-1)}\right)     \left(1\wedge  (v-u)^{2(H-1)}\right)    \left(1\wedge (v-s)^{2(H-1)}\right) \dif s\dif t\dif u \dif v\notag\\
 &\le C \left(\int_{0 }^{\infty} \left(1\wedge s^{2(H-1)}\right) \dif s\right)^4\le C.\label{kappa4 di2}
\end{align}
In the same vein, we have 
\begin{align}
  \abs{ \int_{0\le s<t<u<v \le T} \left(\rho_1(s,t)-\rho_2(s,t)\right) \rho_2(t,u)\left(\rho_1(u,v)-\rho_2(u,v)\right)\rho_2(v,s) \dif s\dif t\dif u \dif v} &<C ,\notag\\
   \abs{\int_{0\le s<t<u<v \le T} \left(\rho_1(s,t)-\rho_2(s,t)\right) \rho_2(t,u)\rho_2(u,v)\left(\rho_1(v,s)-\rho_2(v,s) \right)\dif s\dif t\dif u \dif v} &<C T^H,\notag\\
  \abs{ \int_{0\le s<t<u<v \le T} \rho_2(s,t) \left(\rho_1(t,u)-\rho_2(t,u)\right)\left(\rho_1(u,v)-\rho_2(u,v)\right)\rho_2(v,s) \dif s\dif t\dif u \dif v} &<C ,\notag\\
  \abs{\int_{0\le s<t<u<v \le T} \rho_2(s,t)\left(\rho_1(t,u)- \rho_2(t,u)\right)\rho_2(u,v)\left(\rho_1(v,s)-\rho_2(v,s)\right)\dif s\dif t\dif u \dif v} &<C ,\notag\\
  \abs{\int_{0\le s<t<u<v \le T} \rho_2(s,t) \rho_2(t,u)\left(\rho_1(u,v)-\rho_2(u,v)\right)\left(\rho_1(v,s)-\rho_2(v,s)\right)\dif s\dif t\dif u \dif v} &<C . \label{kappa4 di2-4}
  \end{align}
Step 3. It follows from Proposition~\ref{key prop bijiao} that
\begin{align}
& \abs{ \int_{0\le s<t<u<v \le T} \left(\rho_1(s,t)-\rho_2(s,t)\right) \left(\rho_1(t,u)-\rho_2(t,u)\right)\left(\rho_1(u,v)-\rho_2(u,v)\right)\rho_2(v,s) \dif s\dif t\dif u \dif v}\notag\\
 &\le C \int_{0\le s<t<u<v \le T} \left(1\wedge s^{2(H-1)}\right)    \left(1\wedge t^{2(H-1)}\right)     \left(1\wedge u^{2(H-1)}\right) \left(1\wedge (v-s)^{2(H-1)}\right) \dif s\dif t\dif u \dif v\notag\\
 &\le C \int_{[0,\infty )^3} \left(1\wedge s^{2(H-1)}\right)\left(1\wedge t^{2(H-1)}\right) \left(1\wedge u^{2(H-1)}\right) \dif s\dif t \dif u\int_{u<v \le T} \left(1\wedge(v-u)^{2(H-1)}\right) \dif v\notag\\
 &\le C . \label{kappa4 di3}
\end{align} 
In the same vein, we have
\begin{align}
  \abs{ \int_{0\le s<t<u<v \le T}  \left(\rho_1(s,t)-\rho_2(s,t)\right) \left(\rho_1(t,u)-\rho_2(t,u)\right)\rho_2(u,v)\left(\rho_1(v,s)-\rho_2(v,s) \right)\dif s\dif t\dif u \dif v} &\le C T^H,\notag\\
   \abs{ \int_{0\le s<t<u<v \le T} \left(\rho_1(s,t)-\rho_2(s,t)\right)\rho_2(t,u)\left(\rho_1(u,v)-\rho_2(u,v)\right)\left(\rho_1(v,s)-\rho_2(v,s) \right)\dif s\dif t\dif u \dif v} &\le C T^H,\notag\\
  \abs{ \int_{0\le s<t<u<v \le T} \rho_2(s,t)\left(\rho_1(t,u)-\rho_2(t,u)\right)\left(\rho_1(u,v)-\rho_2(u,v)\right)\left(\rho_1(v,s)-\rho_2(v,s) \right)\dif s\dif t\dif u \dif v} &\le C . \label{kappa4 di3-4}
  \end{align}
Step 4. It follows from Proposition~\ref{key prop bijiao} that
\begin{align}
& \Bigg|\int_{0\le s<t<u<v \le T} \left(\rho_1(s,t)-\rho_2(s,t)\right) \left(\rho_1(t,u)-\rho_2(t,u)\right)\notag\\
&\times \left(\rho_1(u,v)-\rho_2(u,v)\right)\left(\rho_1(v,s)-\rho_2(v,s)\right) \dif s\dif t\dif u \dif v \Bigg|\notag\\
 &\le C \int_{0\le s<t<u<v \le T} \left(1\wedge s^{2(H-1)}\right)    \left(1\wedge t^{2(H-1)}\right)     \left(1\wedge u^{2(H-1)}\right)\left(1\wedge(v-s)^{H-1}\right)\dif s\dif t\dif u \dif v\notag\\
 &\le CT^H.\label{kappa4 zuihou}
\end{align}
Finally, by plugging the nine upper bounds \eqref{kappa4 di1}-\eqref{kappa4 zuihou}  into the left hand side of \eqref{kappa 4 mubiao}, we obtain the inequality \eqref{kappa 4 mubiao}.

Case 2: When $(G)$ is given as in Example~\ref{exmp lizi001}, {the inequality \eqref{rho21ts cha bds-new case} allows us to replace} 
all nine upper bounds in \eqref{rho12cha 00001}-\eqref{r1r2chacha3 } of Step 1  with $CT^{2H}$. Finally, plugging the new nine upper bounds into the left hand side of \eqref{kappa 4 mubiao}, we obtain the inequality \eqref{kappa 4 mubiao} since $H\in (0,\frac12)$.
\end{proof}
  
\section*{Acknowledgements}
The work is supported by National Natural Science Foundation of China (No.11961033,12171410) and the General Project of Hunan Provincial Education Department of China (No. 22C0072).

 \end{document}